\documentclass[letterpaper, 10 pt, conference]{ieeeconf}  

\IEEEoverridecommandlockouts                              
\usepackage{amsfonts}  
\usepackage{amsmath}
\usepackage{amssymb}
\usepackage{xcolor}

\definecolor{darkgreen}{RGB}{0,140,0}

\usepackage{amsthm}

\theoremstyle{plain}
\newtheorem{theorem}{\textbf{Theorem}}
\newtheorem{lemma}{\textbf{Lemma}}
\newtheorem{proposition}{\textbf{Proposition}}
\newtheorem{corollary}{Corollary}

\newtheorem{definition}{\textbf{Definition}}
\newtheorem{assum}{\textbf{Assumption}}

\usepackage{subcaption}
\usepackage{graphicx}

\usepackage[hidelinks]{hyperref}

\newcommand{\scal}{\mathbf{\cdot}}

\title{\LARGE \bf Input-to-Output Stability and Consensus Control of Opinion Dynamics on Graph State-Space}

\author{Arthur Doliveira$^\star$, Christophe Roman$^\star$, Guillaume Graton$^{\star\dagger}$, and Mustapha Ouladsine$^\star$
\thanks{$^{\star}$ A. Doliveira, C. Roman, G. Graton and M. Ouladsine are  with Aix-Marseille Université, Université de Toulon, CNRS, LIS (UMR 7020), Avenue Escadrille Normandie-Niemen, F-13397 Marseille Cedex 20, France
        ({\tt\small  christophe.roman@lis-lab.fr})}%
\thanks{$^{\dagger}$ G. Graton is also with Centrale Méditerranée, Technopôle de Château-Gombert, 38 rue Frédéric Joliot-Curie, F-13451 Marseille Cedex 13, France
        ({\tt\small guillaume.graton@lis-lab.fr })}%
\thanks{\textbf{Acknowledgment}: This project is supported by the French State, in the case of CASSIOPEE and ERMA project, as part of Investments for the Future Programme, now integrated into France 2030 with BPI France, and operated by ADEME.}
}

\begin{document}

\maketitle
\thispagestyle{empty}
\pagestyle{empty}

\begin{abstract}
This paper proposes a modeling of opinion dynamics in social systems with evolving structures on the Graph State-Space. It builds upon the Hegselmann–Krause (HK) model with variable influence weights, integrated into a hybrid dynamical systems framework where structural changes (agent addition/removal) are induced by inputs. We establish an Input-to-Output Stability (IOS) result under a strong connectivity assumption, using a nonsmooth Lyapunov function based on the opinion diameter; we show that the consensus expansion is then bounded by the maximal deviation of incoming agents relative to the existing convex hull. Assuming constant influence weights, we propose a consensus control result via agent injection in the social system. The results of this paper initiate, to the best of our knowledge, the analysis of resilience, stability, and control of dynamic systems on  Graph State-Space.
\end{abstract}

\section{INTRODUCTION}
\noindent Humans continuously form and adapt their opinions across a broad spectrum of topics \cite{shirzadi2025opinion}. These opinions shape their behavior and play a crucial role in numerous major societal challenges, such as global financial crises, urbanization, and climate change \cite{sirbu2016opinion}. Since the advent of online social networks, opinions have arguably never been more actively solicited, disseminated, and, in some cases, manipulated. Understanding the mechanisms of opinion diffusion in society constitutes not only a scientific challenge but also a major democratic issue, insofar as it enables the development of strategies to combat misinformation \cite{shirzadi2025opinion} as well as antisocial behaviors and opinions \cite{varma2018continuous}. Thus, opinion dynamics is a social phenomenon that attracts the interest of researchers in the social sciences, computer science, mathematics, and statistical physics. Opinion formation is a complex process shaped by individual dispositions—such as knowledge and personal judgments—as well as social interactions and external sources, including mass media \cite{shirzadi2025opinion}, \cite{sirbu2016opinion}.

One way to study opinion dynamics is through mathematical models. Despite their simplifying assumptions, these models help identify general principles and key parameters governing opinion diffusion \cite{shirzadi2025opinion}. A social system is typically represented as a graph whose vertices represent individuals carrying one or more numerical opinions, while edges represent social interactions. An evolution law then describes how the system changes over time. The literature generally distinguishes three types of evolution models \cite{shirzadi2025opinion}, \cite{sirbu2016opinion}: discrete, continuous, and hybrid. Discrete models consider opinions in a finite set, such as ${-1,+1}$, as in the Voter \cite{yildiz2013binary} and Cascade \cite{kempe2003maximizing} models. Continuous models allow opinions to vary over an infinite set, as in the French-DeGroot \cite{degroot1974reaching} and Hegselmann-Krause \cite{rainer2002opinion} models. Hybrid models combine both approaches; the CODA (Continuous Opinion and Discrete Action) model \cite{martins2008continuous} represents the degree of adherence to a discrete opinion ${-1,+1}$ through a probability $p_i$ of adopting $+1$ and $1-p_i$ of adopting $-1$.


S. McQuade et al. \cite{mcquade2019social} extended the Hegselmann--Krause model by coupling opinion dynamics with continuous influence-weight dynamics, allowing the representation of group hierarchies and the identification of influential individuals. These models assume that each individual exerts the same influence on all others. N. Ayi and N. Pouradier Duteil \cite{ayi2021mean} studied their large-population limit using mean-field and graph-limit approaches. Here, we consider one of the models introduced by McQuade et al. \cite{mcquade2019social} on an evolving social structure, where individuals and interactions may be added or removed over time. Since network structure influences opinion diffusion \cite{granovetter1973strength}, we study how opinion dynamics evolve as the underlying social structure changes.


To this end, we consider the graph state-space introduced in \cite{doliveira:hal-05047454}, which jointly captures infinitesimal variations of vertex attributes and edge weights, as well as structural changes through its algebraic structure. Modeling the latter as discrete events, we reformulate the model of S. McQuade et al. \cite{mcquade2019social} within the framework of hybrid dynamical systems \cite{goebel2012hybrid}. The resulting dynamics, including the addition and removal of individuals, are then studied through input-to-output stability.

Social equilibria in opinion dynamics—consensus, clustering, or polarization—are typically studied under bounded confidence \cite{bernardo2022finite, bernardo2024bounded}. Stability analysis relies on Lyapunov methods using disagreement energies, opinion diameter, or nonsmooth and functional formulations for switching or delayed systems \cite{amelkin2017polar, etesami2018simple, paolucci2020convergence}. In this paper, we focus on consensus under globally uniform interaction on the Graph State-Space.


Open multi-agent systems have been studied for stability via set-valued jumps on fixed Euclidean spaces \cite{franceschelli2020stability} and for average consensus limits \cite{de2022fundamental}. This work differs in three aspects. First, structural changes are modeled as algebraic operations in a variable semi-vector space. Second, we establish Input-to-Output Stability (IOS) for the nonlinear Hegselmann–Krause dynamics with co-evolving influence weights, under the strong connectivity assumption of the agents. Third, building on this IOS property, we propose a consensus control result grounded in the agent injection strategy under constant influence weights.

Section \ref{GraphStateSpace} introduces the Graph State-Space. Section \ref{HybridMdodel} formulates the HK model with dynamic weights as a hybrid system. Section \ref{NonSmoothLyapunov} establishes IOS for the co-evolving weight model via a nonsmooth Lyapunov function. Section \ref{ControlRes} designs a consensus control strategy via agent injection (constant weights). Section \ref{NumericalTest} presents numerical implementations. Section \ref{Conc} concludes.

\section{Graph State-Space}
\label{GraphStateSpace}
\noindent  Two equivalent definitions of a graph are considered: set-theoretic and vectorial. We begin with the first.  

\subsection{Set-Theoretic Definition}

\noindent Let $\mathcal{I} \subset \mathbb{N}$ be a countable set of labels.  Let $\mathbb{V} = \{(i, x_i) \mid i \in \mathcal{I}\}$ denote the set of individuals, where each individual is represented by its label $i \in \mathcal{I}$ and its opinion value $x_i \in \mathbb{R}$. The set-theoretic definition of graph is as follows: 

\begin{definition}
\label{GraphDef} 
   A  graph $G$  built on $\mathbb{V}$, consists of a pair $(V_G, E_G)$, where $V_G \subseteq \mathbb{V}$ and $E_G$ are the sets of vertices and edges of $G$, respectively, such that: 
 \vspace{-1 mm} 
\[
V_G = \{(i, x_i) \mid i \in \mathcal{I}_G, \; x_i \in \mathbb{R}\}, \]
\vspace{-4 mm}
\[E_G = \{(p, q, w_{p,q}) \mid (p, q) \in \mathcal{L}_G, \; w_{p,q} \in \mathbb{R}\},
\]

\vspace{-1 mm}

\noindent Here, $\mathcal{I}_G \subseteq \mathcal{I}$ and $\mathcal{L}_G \subseteq \mathcal{I}_G \times \mathcal{I}_G$. 
\end{definition}

An arbitrary graph $G$ on $\mathbb{V}$ represents a social system, where $w_{p,q}\in\mathbb{R}$ measures the influence of $q$ on $p$’s opinion. We associate with $G$ the adjacency matrix $A=(a_{p,q})_{(p,q)\in\mathcal{I}G^2}$, where $a{p,q}\in{0,1}$ indicates the presence of a communication channel between $p$ and $q$; we denote by $\mathbb{G}$ the set of all graphs on $\mathbb{V}$.

\begin{assum}
\label{WeightsAssum}
Let $G$ be a graph whose vertices are indexed by $\mathcal{I}_G$. As in \cite{mcquade2019social, ayi2021mean}, we assume that each individual $p \in \mathcal{I}_G$  propagates their opinion with the same weight to every other individual $q \in \mathcal{I}_G$. That is, $w_{p,q} :=w_p<$ $\forall q \in \mathcal{I}_G$.
\end{assum}

\subsection{Vectorial Definition}

 \noindent Let $\mathcal{B}=\{\mathbf{b}_i\mid i\in\mathcal{I}\}$ be a set of linearly independent vectors indexed by the individual labels in $\mathcal{I}$. For any set $\mathcal{S}$, let $2^\mathcal{S}=\{S\mid S\subseteq\mathcal{S}\}$ denote its power set. Let $\mathcal{B}_e=\mathcal{B}\otimes\mathcal{B}$ be the tensor product \cite{burdick1995introduction}. For $B\in2^\mathcal{B}$, define $I(B)=\{i\in\mathcal{I}\mid\mathbf{b}_i\in B\}$, and for $D\in2^{\mathcal{B}_e}$, define $I(D)=\{(k,l)\in\mathcal{I}^2\mid\mathbf{b}_k\otimes\mathbf{b}_l\in D\}$. The subspaces associated with $B$ and $D$ are defined as follows:

 \vspace{-4 mm}

\begin{equation}
\label{attVector}
\mathcal{V}(B)
= \operatorname{span}(B)
= \left\{
\sum_{i\in I(B)} \alpha_i\mathbf{b}_i
\;\middle|\;
\alpha_i\in\mathbb{R}
\right\}.
\end{equation}

 \vspace{-2 mm}

and

 \vspace{-7 mm}
 
\begin{equation}
\label{boolVect}
\begin{split}
\mathcal{A}(D)
&=
\left\{
\sum_{(k,l)\in I(D)}
\gamma_{k,l}\mathbf{b}_k\otimes\mathbf{b}_l
\;\middle|\;
\gamma_{k,l}\in\{0,1\}
\right\}.
\end{split}
\end{equation}

As in \cite{doliveira:hal-05047454, doliveira2026generalized}, we define the variable-basis spaces as
\begin{equation}
\mathcal{V} := \bigcup_{B \in 2^{\mathcal{B}}} \mathcal{V}(B),
\qquad
\mathcal{A} := \bigcup_{D \in 2^{\mathcal{B}_e}} \mathcal{A}(D).
\end{equation}

An element $\mathbf{x}_B \in \mathcal{V}$ and an element $\mathbf{a}_D \in \mathcal{A}$, with $B \in \mathcal{B}$ and $D \in \mathcal{B}_e$, are represented with respect to the bases $B$ and $D$, respectively, as
\begin{equation}
\mathbf{x}_B\! :=\!\!\sum_{k \in I(B)} x_k \mathbf{b}_k,
\qquad
\mathbf{a}_D\!:=\!\!\! \sum_{(p,q) \in I(D)}\!\!\! a_{p,q}
\mathbf{b}_p \otimes \mathbf{b}_q,
\end{equation}
where $x_k \in \mathbb{R}$ and $a_{p,q} \in \{0,1\}$. Let
\begin{equation}
\label{DB_Space}
    \mathcal{D}(B) := \mathcal{V}(B)^2 \times \mathcal{A}(D),
~\text{ with }~ D = B \otimes B.
\end{equation}
We define the variable-basis space $\mathcal{D}_{\mathbb{G}}
:= \bigcup_{B \in 2^{\mathcal{B}}} \mathcal{D}(B)$. An element of $\mathcal{D}_{\mathbb{G}}$ is written as $\mathbf{X}_B := (\mathbf{x}_B,\mathbf{w}_B,\mathbf{a}_D)$ with $D := B \otimes B$ for $B \in 2^{\mathcal{B}}$. Under Assumption \ref{WeightsAssum}, we consider the following proposition:

\begin{proposition}
The sets $\mathbb{G}$ and $\mathcal{D}_{\mathbb{G}}$ are bijective.
\end{proposition}

The proof is given in Section 3 of \cite{doliveira:hal-05047454}, where Assumption~\ref{WeightsAssum} is not considered; including it does not affect the demonstration. Under Assumption~\ref{WeightsAssum}, the graph is defined in vector form as follows:

\begin{definition}
Let $B \in \mathcal{B}$. A graph $X_B \in \mathcal{D}_{\mathbb{G}}$ is defined as the triplet $X_B = (\mathbf{x}_B, \mathbf{w}_B, \mathbf{a}_D),$ where:  \emph{(i)} $\mathbf{x}_B \in \mathcal{V}$ is the vector of opinions of individuals $i \in I(B)$; \emph{(ii)} $\mathbf{w}_B \in \mathcal{V}$ is the vector of influences on the opinions of individuals $p \in I(B)$ coming from other individuals $q \in I(B)$;  \emph{(iii)} $\mathbf{a}_D \in \mathcal{A}$ is the adjacency matrix, with $D = B \otimes B$.
\end{definition}

\subsection{Algebraic Structure}
\label{AlgbraicStruct}
\noindent An algebraic structure establishes relations between the elements of a set via algebraic operations \cite{bourbaki1998algebra}. We consider three of them: two internal composition laws $+_\cup$ and $+_\cap$, and a multiplicative external law, introduced in \cite{doliveira:hal-05047454} and defined on $\mathcal{D}_{\mathbb{G}}$. Let $X_B, Y_C \in \mathcal{D}_{\mathbb{G}}$ such that $X_B = (\mathbf{x}_B, \mathbf{w}_B, \mathbf{a}_D)$ with $D = B \otimes B$ and $Y_C = (\mathbf{y}_C, \boldsymbol{\omega}_C, \mathbf{e}_L)$ with $L = C \otimes C$.  For all $\star \in \{\cup,\cap\}$, the internal laws $+_\star$ are defined as follows:
\begin{equation}
    X_B +_\star Y_C = (\mathbf{x}_B +_\star \mathbf{y}_C, \mathbf{w}_B +_\star \boldsymbol{\omega}_C, \mathbf{a}_D \diamond_\star \mathbf{e}_L),
\end{equation}
where $\diamond$ is a variant of $+$ in Boolean logic: the disjunction $\vee$ if $\star = \cup$, the conjunction $\wedge$ if $\star = \cap$. These operations satisfy the fundamental axioms of Boolean algebra \cite{halmos2009introduction}, the most relevant of which are: for all $\lambda \in \{0,1\}$
\begin{equation}
    \lambda \vee 1 = 1, \quad \lambda \vee 0 = \lambda, \quad \lambda \wedge 1 = \lambda, \quad \lambda \wedge 0 = 0. 
\end{equation}
 For $\star \in \{\cup,\cap\}$, we define the operations $+_\star$ on $\mathcal{V}$ and $\diamond_\star$ on $\mathcal{A}$, with $\diamond \in \{\vee, \wedge\}$, as in \cite{doliveira:hal-05047454}:

 \begin{definition}
 \label{realop}
The operation \(+_\star\) on \(\mathcal{V}\), with \(\star \in \{\cup, \cap\}\), is a mapping \(+_\star : \mathcal{V} \times \mathcal{V} \to \mathcal{V}\); for any pair \((\mathbf{x}_{B}, \mathbf{y}_{C}) \in \mathcal{V}^2\), where
\vspace{-2.5 mm}
\[
 \mathbf{x}_{B} = \sum_{k \in I(B)} x_k \mathbf{b}_k \quad \text{and} \quad \mathbf{y}_{C} = \sum_{l \in I(C)} y_l \mathbf{b}_l,
 \]
the operation \(+_\star\) is given by: 
\[ 
\mathbf{x}_{B} +_{\star} \mathbf{y}_{C} = \mathbf{x}_{B \star C} + \mathbf{y}_{C \star B}, \quad \text{with } \star \in \{\cup, \cap\}, 
\]
where \(\mathbf{x}_{B \star C}\) and \(\mathbf{y}_{C \star B}\) denote the projections of \(\mathbf{x}_{B}\) and \(\mathbf{y}_{C}\) onto the basis \(B \star C = C \star B\). More explicitly, \[ \mathbf{x}_{B} +_{\star} \mathbf{y}_{C} = \sum_{p \in I(B \star C)} (x_p + y_p) \mathbf{b}_p, \quad \text{with } \star \in \{\cup, \cap\}. \] If \(\star = \cup\), then for any \(p \in I(B)\) such that \(p \notin I(C)\) (and vice versa), \(y_p = 0\) (or \(x_p = 0\), respectively). 
 \end{definition}

 \begin{definition}
 \label{boolOp}
The operation \(\star \in \{\vee_\cup, \wedge_\cap\}\) on \(\mathcal{A}\) is a mapping \(\star : \mathcal{A} \times \mathcal{A} \to \mathcal{A}\); for any \((\boldsymbol{\gamma}_{D}, \boldsymbol{\delta}_{L} )\in \mathcal{A}^2\), such that  
\vspace{-2 mm}\[
\boldsymbol{\gamma}_{D} = \sum_{(i,j) \in I(D)} \gamma_{i,j} \mathbf{b}_i \otimes \mathbf{b}_j, \quad \boldsymbol{\delta}_{L} = \sum_{(k,l) \in I(L)} \delta_{k,l} \mathbf{b}_k \otimes \mathbf{b}_l,
\]  

\begin{itemize}  
  \item the disjunction $(\vee_\cup)$ is given by:  
  \[
  \boldsymbol{\gamma}_{D} \vee_\cup \boldsymbol{\delta}_{L} = \sum_{(p,q) \in I(D \cup L)} (\gamma_{p,q} \vee \delta_{p,q}) \mathbf{b}_p \otimes \mathbf{b}_q.
  \]  
When, for any \( (p,q) \in I(D) \) such that \( (p,q) \notin I(L) \) (and vice versa), \( \delta_{p,q} = 0 \) (or \( \gamma_{p,q} = 0 \), respectively).

\vspace{1 mm}

  \item the conjunction $(\wedge_\cap)$ is given by: 
  \vspace{-2 mm} 
  \[
  \boldsymbol{\gamma}_{D} \wedge_\cap \boldsymbol{\delta}_{L} = \sum_{(p,q) \in I(D \cap L)} (\gamma_{p,q} \wedge \delta_{p,q}) \mathbf{b}_p \otimes \mathbf{b}_q.
  \]  
\end{itemize} 
 \end{definition}

 Definitions \ref{realop} and \ref{boolOp} simply state that the operations $+_\star$ and $\diamond_\star$, with $\star \in \{\cup, \cap\}$ and $\diamond \in \{\vee, \wedge\}$, have a twofold action. Operations $\star$ acts on the structure, specifically on the vector bases—union generally enlarges the structure, intersection reduces it—while $+$ and $\diamond$ act on the coordinates.

\begin{definition}
The external scalar multiplication is the map $\scal:\mathbb{R}\times\mathcal{D}_{\mathbb{G}}\to\mathcal{D}_{\mathbb{G}}$, $(\alpha,X_B)\mapsto\alpha\scal X_B$, where, if $X_B=(\mathbf{x}_B,\mathbf{w}_B,\mathbf{a}_D)$ with $D=B\otimes B$, then $\alpha\scal X_B:=(\alpha\mathbf{x}_B,\alpha\mathbf{w}_B,\mathbf{a}_D)$.
\end{definition}

\begin{theorem}
\label{theocup}
The triple $(\mathcal{D}_{\mathbb{G}}, +_\cup, \scal)$ is a complete and simple semi-vector space over $\mathbb{R}$.
\end{theorem}

\begin{theorem}
\label{theocap}
The triple $(\mathcal{D}_{\mathbb{G}}, +_\cap, \scal)$ is a complete and simple semi-vector space over $\mathbb{R}$.
\end{theorem}

A semi-vector space is said to be complete if it possesses a neutral element, and simple if it contains elements without an inverse \cite{janyvska2007semi}. The proofs of theorems \ref{theocup} and \ref{theocap} are provided in \cite{doliveira:hal-05047454}. In this work, the algebraic structure under consideration is $(\mathcal{D}_{\mathbb{G}}, +_\cup,+_\cap, \scal)$, for which theorems \ref{theocup} and \ref{theocap} hold. The space $\mathcal{D}_{\mathbb{G}}$, endowed with the distance defined in \cite[Subsection 3.3]{doliveira:hal-05047454}, is a metric space.

\section{Social model in hybrid dynamics}
\label{HybridMdodel}
\noindent The evolution of a social system in the space $\mathcal{D}_{\mathbb{G}}$, whose behavior is both continuous — through opinions and influences — and discrete — through the structure, or even through opinions and influences — can be studied within the framework of hybrid dynamical systems theory. To formulate the model in this framework, we adopt the formalism of Goebel et al. \cite{goebel2012hybrid}. We begin by introducing the necessary elements for its construction.

In the hybrid systems framework of Goebel et al. \cite{goebel2012hybrid}, the state is parameterized by continuous time \(t\in\mathbb{R}_{\geq0}\) and jump index \(k\in\mathbb{N}\), with hybrid time domain \(\operatorname{dom}(X)\subset\mathbb{R}_{\geq0}\times\mathbb{N}\) (see \cite{doliveira:hal-05047454,goebel2012hybrid}). Following \cite{doliveira:hal-05047454}, the social-system state at \((t,k)\in\operatorname{dom}(X)\) is \(X_{B_k}(t,k):=\big(\mathbf{x}_{B_k}(t,k),\mathbf{w}_{B_k}(t,k),\mathbf{a}_{D_k}(t,k)\big)\in\mathcal{D}_{\mathbb{G}}\), where \(B_k\subseteq\mathcal{B}\) and \(D_k=B_k\otimes B_k\); for simplicity, \(X_{B_k}(t):=X_{B_k}(t,k)\). The perturbation is an essentially bounded map \(U_C:\mathbb{R}_{\geq0}\to\mathcal{D}_{\mathbb{G}}\), given by \(U_{C_t}(t):=\big(\mathbf{u}_{C_t}(t),\boldsymbol{\omega}_{C_t}(t),\mathbf{e}_{L_t}(t)\big)\), where \(C:\mathbb{R}_{\geq0}\to2^{\mathcal{B}}\), \(C_t=C(t)\), and \(L_t=C_t\times C_t\). We assume that the perturbation is known at each time and affects \(\operatorname{dom}(X)\) \cite{heemels2026solution}. The evolution model is given by:
 \begin{equation} \small
   \label{GenEquaDiff}
    \!\!\! \left\{
         \begin{aligned}
           &\dot{X}_{B_k}\in f\left(X_{B_k}\right), \qquad ~~~~~X_{B_k}  \in \! F,   \\
           &X_{B_{k+1}}^+ = j\left(X_{B_k}^-,U_{C_t}^-\right),\quad  (X_{B_k}^-,  U_{C_{t}}^-) \!\in\! J ,
        \end{aligned}
      \right.
\end{equation}
where \(F\subset\mathcal{D}_{\mathbb{G}}\) and \(f:\mathcal{D}(B)\to2^{\mathcal{D}(B)}\), with \(\mathcal{D}(B)\subset\mathcal{D}_{\mathbb{G}}\) for all \(B\subset\mathcal{B}\), are the flow set and map, respectively, while \(J\subset\mathcal{D}_{\mathbb{G}}\times\mathcal{D}_{\mathbb{G}}\) and \(j:\mathcal{D}_{\mathbb{G}}\times\mathcal{D}_{\mathbb{G}}\to\mathcal{D}_{\mathbb{G}}\) are the jump set and map. \(X_{B_k}^{-}\) and \(U_{C_t}^{-}\) denote the left limits at time \(t\), where \((t,k),(t,k+1)\in\operatorname{dom}(X)\), and \(X_{B_{k+1}}^{+}\) the corresponding right limit. The hybrid system is denoted by \(\mathcal{H}:=\mathcal{H}(F,f,J,j)\). The flow map \(f\), which governs the continuous dynamics, is the set-valued map \(f=\mathcal F[f_o,f_i,f_a]\), where \(\mathcal F\) is the operator defined in \cite{hajek1979discontinuous}, allowing us to consider Filippov solutions as in \cite[Def. 1 and Prop. 1]{mcquade2019social}. For all \(B\subseteq\mathcal B\), \(f_o,f_i:\mathcal D(B)\to\mathcal V(B)\) and \(f_a:\mathcal D(B)\to\{\mathbf 0_D\}\subset \mathcal A(D)\), where \(D=B\otimes B\). Following Model 4 in \cite{mcquade2019social}, the dynamics for an agent \(p\in I(B)\) are given by:
\begin{equation}
\label{continuousDynamic}
    \dot{x}_p = f_o^p(X_B), \quad \dot{w}_p = f_i^p(X_B), \quad \dot a_{p,q}=f_a^{p,q}(X_B)
\end{equation}
with
\begin{equation}\label{dynamique}
 \!\!\!   \begin{cases} 
\dot x_p  = \frac{1}{W}\displaystyle\sum_{q\in I(B)} w_q \sigma_{p,q}(x_q-x_p),\\[1em]
\dot w_p= w_p\!\left(\frac{1}{W}\!\! \displaystyle\sum_{q\in I(B)} \!\! w_q\frac{\beta_{q,p}}{\frac{1}{W}\!\!\sum\limits_{s\in I(B)}\!\! w_s \beta_{q,s}}-1\right) \!,\\
\dot a_{p,q}=0,
\end{cases}
\end{equation}
where $W = \sum_{l\in I(B)} w_l$ is the total mass, $\sigma_{p,q} = a_{p,q}\sigma(|x_p - x_q|)$, with $\sigma(\cdot)$ denoting the interaction function. The term $a_{p,q}$ of the adjacency matrix indicates the presence of a communication channel between individuals $p$ and $q$. The term $\beta_{q,p} = w_p\sigma_{q,p}|x_p - x_q|$ represents the influence of individual $p$ on $q$. As the model naturally exhibits discontinuities in the weight dynamics at consensus, solutions between jumps are considered in the Filippov sense \cite[Def.~1 and Prop. 1]{mcquade2019social}, thereby yielding the differential inclusion in \eqref{GenEquaDiff}.

The jump map $j$ governs the discrete dynamics affecting the opinions, the influence weights, and the system structure. These structural changes, induced by the perturbation $U_{C}$, result in the addition or removal of individuals along with their connections. Assuming that $J = J_{\nearrow} \cup J_{\searrow}$, the jump map $j$ can be expressed, as in \cite{doliveira:hal-05047454}, as follows:
\begin{equation}
\label{Defjump}
j(X_{B_k}^-,U_{C_t}^-) \!=\!
   \begin{cases}
      X_{B_k}^- \!+_\cup\! U_{C_t}^-, & \text{if } (X_{B_k}^-,U_{C_t}^-)\! \in\! J_{\nearrow},\\
      X_{B_k}^- \!+_\cap\! U_{C_t}^-, & \text{if } (X_{B_k}^-,U_{C_t}^-)\! \in\! J_{\searrow},
  \end{cases}
\end{equation}

\vspace{-1 mm}

where,
\begin{equation} \small
\label{jumpset}
\begin{aligned}
J_{\nearrow} &= \left\{(X_{B},Y_{C}) \in \mathcal{D}_\mathbb{G}^2 \mid C\neq \varnothing, \exists H \subseteq C \text{ s.t. } H \not\subset B\right\}, \\
J_{\searrow} &= \left\{(X_{B},Y_{C}) \in \mathcal{D}_\mathbb{G}^2 \mid C\neq \varnothing \text{ s.t. } C \subset B\right\}.
\end{aligned}
\end{equation}
The operations $+_\cup$ and $+_\cap$, defined in Section \ref{AlgbraicStruct} (see also \cite{doliveira:hal-05047454}), respectively allow the addition of new individuals to the social system and the removal of others, specifying in the input $U_{C_t}$ which ones should be retained. 

\section{Resilience Analysis via Nonsmooth Lyapunov Functions}
\label{NonSmoothLyapunov}
\noindent The social system, subject to abrupt perturbations such as the addition or removal of individuals, is analyzed in terms of resilience, particularly with respect to consensus.  We propose to analyze the system using a nonsmooth Lyapunov function.

\subsection{Non Smooth Lyapunov Function}
\noindent Consider the following Lyapunov function, from \cite{amelkin2017polar}:
\begin{align}\label{def_V}
    V(X_B)\!=\!\max_{p,q\in I(B)}|x_p-x_q|\!=\!\max_{p\in I(B)}x_p-\min_{q\in I(B)}x_q.
\end{align}
In a fixed basis space $\mathcal{D}(B) \subset F$, for all $B \in 2^{\mathcal{B}}$ (see \eqref{DB_Space}), the function $V$ is convex, hence locally Lipschitz and possessing a generalized gradient. For the stability analysis of nonsmooth Lyapunov candidates, we employ tools from nonsmooth analysis \cite{amelkin2017polar, clarke1990optimization, cortes2008discontinuous}. Considering \eqref{def_V}, for a given state $X_{B_k}(t)$, we define the sets of indices achieving the extrema as:
\begin{align}
 I_\text{M}(X_{B_k}(t)) &= \{ p \mid x_p(t,k) = \max_{q \in I(B_k)} x_q(t,k) \},\\
I_\text{m}(X_{B_k}(t)) &= \{ p \mid x_p(t,k) = \min_{q \in I(B_k)} x_q(t,k) \}.   
\end{align}
 The Clarke generalized gradient of $V(X_{B})$ is the convex hull of the gradients of the active smooth functions: 
    \begin{align}
     \partial V(X_B) \!=\! \text{co} \{ \mathbf{b}_p - \mathbf{b}_q \mid p \in I_\text{M}(X_B), q \in I_\text{m}(X_B) \},   
    \end{align}
    where $\mathbf{b}_k$ denotes the k-th basis vector of the subset $B \subset \mathcal{B}$. Any element $\zeta \in \partial V(X_B)$ can be expressed as a convex combination:
    \begin{align}
     \zeta = \!\!\sum_{p \in I_{\text{M}}, q \in I_{\text{m}}}\!\! \mu_{p,q} (\mathbf{b}_p - \mathbf{b}_q),  \text{ with } \sum \mu_{p,q} = 1, ~ \mu_{p,q} \geq 0.  \notag  
    \end{align}
    The set-valued Lie derivative is then given by the set:
    \begin{align}
     \mathcal{L}_{f_o} V(X_B) = \left\{ \langle\zeta,f_o(X_B)\rangle  \mid \zeta \in \partial V(X_B) \right\},  \notag
    \end{align}
    where $f_o$ governs the opinion dynamics (see \eqref{continuousDynamic}).
\begin{definition}
\label{expansionmagnitude}
Let $X_B,U_C\in\mathcal D_{\mathbb G}$, with $\overline{x}=\max_{p\in I(B)}x_p$, $\underline{x}=\min_{p\in I(B)}x_p$, $\overline{u}=\max_{p\in I(C)}u_p$, and $\underline{u}=\min_{p\in I(C)}u_p$. We define the convex hulls $H_X=[\underline{x},\overline{x}]$ and $H_U=[\underline{u},\overline{u}]$. We denote by $H_X+H_U$ the Minkowski sum of $H_X$ and $H_U$, defined by $H_X+H_U=\{x+u\mid x\in H_X,\ u\in H_U\} =[\underline{x}+\underline{u},\overline{x}+\overline{u}]$. We then define the maximum possible expansion of the opinion diameter induced by the perturbation $U_C$ as
\[
\Gamma_{X_B}(U_C)=\operatorname{diam}\bigl(H_X\cup H_U\cup(H_X+H_U)\bigr)-\operatorname{diam}(H_X),
\]
\vspace{-4 mm}
\[
\begin{split}
    \text{where }~\operatorname{diam}&\bigl(H_X\cup H_U\cup(H_X+H_U)\bigr)=\\
    &\max\{\overline{x},\overline{u},\overline{x}+\overline{u}\}-\min\{\underline{x},\underline{u},\underline{x}+\underline{u}\},
\end{split}
\]
and $\operatorname{diam}(H_X)=\overline{x}-\underline{x}$. If $I(B)\cap I(C) = \varnothing$, the Minkowski sum $H_X +H_U$ is empty.
\end{definition}
\begin{theorem}\label{ISS}
Consider the opinion dynamics $\mathcal{H}$ given by \eqref{GenEquaDiff}, which is strongly connected, \textit{i.e.}, $\alpha = \min_{p,s,q}\left\{w_q\sigma_{p,q},w_q\sigma_{s,q}\right\}>0$. Then, it holds that
\vspace{-4 mm}
\[
\max \mathcal{L}_{f_o} V(X_{B_k}) \leq -\eta V(X_{B_k}), ~~\forall X_{B_k}\in F;
\]
\[
V\bigl(X_{B_{k+1}}^+\bigr)-V\bigl(X_{B_k}^-\bigr)
\leq \Gamma_{X_{B_k}}\bigl(U_{C_t}^-\bigr), ~~\forall \bigl(X_{B_k}^-,U_{C_t}^-\bigr)\in J,
\]
with \(\Gamma_{X_{B_k}}\bigl(U_{C_t}^-\bigr)\) as defined in Definition~\ref{expansionmagnitude}.
\end{theorem}
\begin{lemma}\label{lem_flow}
    Consider the continuous dynamics of the system $\mathcal{H}$ (see \eqref{GenEquaDiff} and \eqref{dynamique}) within  $\Gamma \subset \mathcal{D}(B) \cap F$, i.e., between jumps. 
If $\min_{p,s,q} \{w_q \sigma_{p,q}, w_q \sigma_{s,q}\} > 0$, then it holds that:
  \begin{align}
       \max \mathcal{L}_{f_o} V(X_{B_k})\leq -\eta V(X_{B_k}).
  \end{align}
\end{lemma}
\begin{proof}   To establish exponential stability, it suffices to show that for each pair of critical indices $(p, s)$, the dynamics of their distance is bounded above by a negative multiple of the maximal distance. We have:
    \begin{align}
        \langle \xi,f_o(X_{B_k})\rangle = \sum_{p,s} \mu_{ps} (\dot{x}_p - \dot{x}_s),
    \end{align}
    \begin{align}
    \dot{x}_p \!-\! \dot{x}_s \!=\! - \frac{1}{W}\bigg[\!\sum_q w_q \sigma_{p,q}(x_p \!-\! x_q) \!+\!\sum_q w_q \sigma_{s,q}(x_q \!-\! x_s) \bigg]. \notag   
    \end{align}
    We can bound this expression as follows:
    \begin{align}
     \dot{x}_p - \dot{x}_s \leq -\frac{\alpha}{W} \sum_q \Big[ (x_p - x_q) + (x_q - x_s) \Big],   \\
    \dot{x}_p - \dot{x}_s \leq -\frac{\alpha}{W} \sum_q V(X_{B_k}) = -\frac{\alpha N}{W} V(X_{B_k}).    
    \end{align}
    Let $\lambda = \frac{\alpha N}{W} > 0$, with $\alpha = \min_{p,s,q} \{w_q \sigma_{p,q}, w_q \sigma_{s,q}\}$, be the convergence rate. Since for every pair $(p, s)$ we have $(\dot{x}_p - \dot{x}_s) \leq -\lambda V(X_{B_k})$, then for any element $\xi$ of the set-valued Lie derivative (which is a convex combination of these terms), it holds that
    \begin{align}
     \langle \xi,f_o(X_{B_k})\rangle \leq \sum_{p,s} \mu_{ps} (-\lambda V(X_{B_k})) = -\lambda V(X_{B_k}) . 
    \end{align}
    Thus, $\max \mathcal{L}_{f_o} V(X_{B_k}) \leq -\lambda V(X_{B_k})$.
\end{proof}

\begin{lemma}\label{lem_jump}
Consider the opinion dynamics of the system $\mathcal{H}$ (see \eqref{GenEquaDiff}) at a jump time $t$ such that $(t,k)\in\operatorname{dom}(X)$ and $(t,k+1)\in\operatorname{dom}(X)$. Then,
\begin{align}
V\bigl(X_{B_{k+1}}^+\bigr)-V\bigl(X_{B_k}^-\bigr)
\leq \Gamma_{X_{B_k}^-}\bigl(U_{C_t}^-\bigr),
\end{align}
where $X_{B_k}^-$ and $U_{C_t}^-$ denote the left limits at $t$, and $X_{B_{k+1}}^+$ denotes the right limit at $t$.
\end{lemma}

\begin{proof}
Let \(Z=X_{B_{k+1}}^+=j\bigl(X_{B_k}^-,U_{C_t}^-\bigr)\), where \(j\) is defined in \eqref{Defjump}. For notational convenience, let \(I_X=I(B_k)\) and \(I_U=I(C_t)\). The worst-case expansion induced by \(+_{\cup}\) gives
\[
\overline{z}
=
\max\left\{
\max_{i\in I_X\setminus I_U}x_i,\,
\max_{i\in I_U\setminus I_X}u_i,\,
\max_{i\in I_X\cap I_U}(x_i+u_i)
\right\},
\]
and
\[
\underline{z} = \min\left\{
\min_{i\in I_X\setminus I_U}x_i,\,
\min_{i\in I_U\setminus I_X}u_i,\,
\min_{i\in I_X\cap I_U}(x_i+u_i)
\right\}.
\]
Hence,
\[
V\bigl(X_{B_{k+1}}^+\bigr)
=
\overline{z}-\underline{z}
\leq
\operatorname{diam}\bigl(H_X\cup H_U\cup(H_X+H_U)\bigr),
\]
where \(H_X\) and \(H_U\) denote the convex hulls of the opinions in \(X_{B_k}^-\) and \(U_{C_t}^-\), respectively, and \(H_X+H_U\) denotes their Minkowski sum. Since $V\bigl(X_{B_k}^-\bigr) = \overline{x}-\underline{x} = \operatorname{diam}(H_X)$, it follows that
\[
V\bigl(X_{B_{k+1}}^+\bigr)
-
V\bigl(X_{B_k}^-\bigr)
\leq
\Gamma_{X_{B_k}^-}\bigl(U_{C_t}^-\bigr).
\]
\end{proof}

The proof of Theorem~\ref{ISS} is the combination of Lemma~\ref{lem_flow} and Lemma~\ref{lem_jump}.  Theorem~\ref{ISS} relies on a strong connectivity assumption on the continuous dynamics, which is relatively restrictive. This assumption could, however, be relaxed using additional tools. In particular, \cite[Thm.~2]{moreau2004stability} could be used for this purpose.


\section{\label{ControlRes}Consensus Control through Agent Injections}

\noindent In this section, we build on Theorem~\ref{ISS} to establish a first control result aimed at achieving a consensus $x^*$ through agent injections. We restrict our study to the following subset:
\[
\mathcal{S}:=\left\{X_B \!\in\!\mathcal{D}_{\mathbb G}\mid
\forall p,q \!\in\! I(B),\ x_p\geq0,\ w_p\geq0,\ a_{p,q}\!=\!1\right\}.
\]

\begin{assum}\label{assumption:control}
We assume that, for \(X_{B_k}=(\mathbf{x}_{B_k},\mathbf{w}_{B_k},\mathbf{a}_{D_k})\) evolving according to \eqref{GenEquaDiff}-\eqref{jumpset},
\begin{enumerate}
    \item \(w_p(t,k)=c_p\in[\underline{w},\overline{w}]\) \(\forall  p\in I(B_k)\), with $0 < \underline{w} \le \overline{w} < \infty$;
    
    \item the interaction function satisfies \(\sigma(r)=1\), \(\forall r\in\mathbb{R}\);
    
    \item the jump map \(j\) in \eqref{Defjump} is restricted to its \(J_{\nearrow}\) component: \(X_{B_{k+1}}^+=X_{B_k}^-+_{\cup}U_{C_t}^-\);
    
    \item the incoming agents introduced through \(U_{C_t}=\bigl(\mathbf{u}_{C_t},\boldsymbol{\omega}_{C_t},\mathbf{e}_{L_t}\bigr)\) do not perturb the existing agents, i.e., \(u_p=\omega_p=0\) for all \(p\in I(B_k)\cap I(C_t)\).
\end{enumerate}
\end{assum}

 Define the weighted-average functional $M:\mathcal{S}\to\mathbb{R}$ such that 
\begin{equation}
\label{averagefunc}
    M(X_B):=
\frac{\sum_{p\in I(B)}w_px_p}{\sum_{p\in I(B)}w_p}
=
\frac{\langle\mathbf w_B,\mathbf x_B\rangle_{\mathcal V(B)}}
{\langle\mathbf w_B,\mathbf 1_B\rangle_{\mathcal V(B)}}.
\end{equation}
\begin{lemma}
\label{moyenneassum}
Let \(X_B,U_C\in\mathcal S\). Under Assumption~\ref{assumption:control}-(4),
\[
 M(X_B+_{\cup}U_C)
=
\frac{W_X M(X_B)+W_U  M(U_C)}{W_X+W_U},
\]
where \(W_X:=\sum_{p\in I(B)}w_p\) and \(W_U:=\sum_{p\in I(C)}\omega_p\).
\end{lemma}

\begin{proof}
The proof follows directly from the definitions of \(+_{\cup}\) in Subsection~\ref{AlgbraicStruct}, under Assumption~\ref{assumption:control}-(4), and of the weighted-average functional \(\mathcal M\) in \eqref{averagefunc}.
\end{proof}

\begin{corollary}
\label{constaverage}
Under Assumptions \ref{assumption:control}-(1) and \ref{assumption:control}-(2), 
the continuous dynamics over $[t_k,t_{k+1}]$ satisfy
\[
\frac{d}{dt}M(X_{B_k})
= \frac{1}{W}\sum_p w_p\frac{d}{dt}x_p
= 0.
\]
Therefore, $M(X_{B_k})$ is invariant over $[t_k,t_{k+1}]$.
\end{corollary}
\begin{proof}
Substituting $\frac{d}{dt}x_p$ from \eqref{dynamique} and using Assumptions \ref{assumption:control}-(1) and \ref{assumption:control}-(2), $\frac{d}{dt}\mathcal{M}(X_{B_k})=\frac{1}{W}\sum_{p,q}w_pw_q(x_q-x_p)=0$ by antisymmetry in $p$ and $q$.
\end{proof}

\begin{theorem}
\label{resControl}
Consider the hybrid system $\mathcal H$ (see \eqref{GenEquaDiff}--\eqref{jumpset}) on $\mathcal S$ under Assumption \ref{assumption:control}. Let the control input sequence $\{U_k \}_{k=1}^{\infty}$, with $U_k:=U_{C_{t_k}}^- \in \mathcal S$, satisfy: (i) $ M(U_k)=x^*$ and $V(U_k)=0$ for all $k$; (ii) $\sum_{k=1}^{\infty}W_k^U=\infty$, where $W_k^U = \sum_{p\in I(C_{t_k})}\omega_p$; (iii) for every pair of successive jump times $t_k,t_{k+1}$, there exist $0<T_{\min}\leq T_{\max}$ such that $T_{\min}\leq t_{k+1}-t_k\leq T_{\max}$. Then, for every initial state $X_{B_0}(0)\in\mathcal S$,

\vspace{-4mm}

\[
\lim_{(t,k)\to\infty} M(X_{B_k}(t))=x^*,
\qquad
\lim_{(t,k)\to\infty}V(X_{B_k}(t))=0.
\]
\end{theorem}

\begin{proof}
Let us begin by fixing the notation. Let $t_k$, with $k\geq 1$, be the jump times such that $(t_k,k),(t_k,k+1)\in\operatorname{dom}(X)$. 
We set $X_k := X^-_{B_k} \in \mathcal S$, with $X_0 := X_{B_0}(0)$ for the initial flow. 
Define $M_k:= M(X_k)$, $V_k:=V(X_k)$, and $W_k:=\sum_{I(B_k)}w_p$.
Assumption~\ref{assumption:control} holds. 
The hybrid dynamics induce the following discrete recursion for $k\ge 1$: (a) jump at $(t_k,k)$: $X^+_{B_{k+1}} = j(X_k, U_k) = X_k +_\cup U_k$ (see Eq.~\eqref{Defjump}); (b) flow from $(t_k,k+1)$: $X^+_{B_{k+1}} \to X_{k+1}$ on basis $B_{k+1}$ over $[t_{k}, t_{k+1})$.

 Let us show that $M_k\to x^*$. Corollary~\ref{constaverage} gives that $M$ is constant along flows, which implies $M(X^+_{B_{k+1}}) = M_{k+1}$.
Lemma~\ref{moyenneassum} applied to the jump yields:
$M_{k+1} = \frac{W_k M_k + W^U_k x^*}{W_k + W^U_k}$. Let $\varepsilon_{k+1} = M_{k+1}-x^*$. Then
$\varepsilon_{k+1} = \frac{W_k}{W_k+W^U_k}(M_k-x^*)$. Thus, we have $\varepsilon_{k+1} = \frac{W_k}{W_k+W^U_k} \varepsilon_k$. By induction, we obtain $\varepsilon_{k+1} = \frac{W_0}{W_k+W^U_k} \varepsilon_0$. Since $\sum_{k=1}^\infty W^U_k=\infty$ by Assumption (ii) of Theorem~\ref{resControl}, it follows that $M_k\to x^*$.

 Let us show that $V_k\to0$. We first examine the expansion at the jump time $t_k$. By Assumption (i) of Theorem~\ref{resControl}, we have $V(U_k)=0$, which means that one or more agents are added with opinion $x^*$. Assumption~\ref{assumption:control}-(4) gives $H_{X^+_{B_{k+1}}}=H_{X_k}\cup\{x^*\}$. Thus, Lemma~\ref{lem_jump} gives:
$V(X^+_{B_{k+1}}) - V_k \le \operatorname{dist}(x^*,H_{X_k}) \le |x^*-M_k|$ since $M_k\in H_{X_k}$. We now consider the contraction over $[t_{k},t_{k+1})$. Lemma~\ref{lem_flow} gives $\max L_{f_o}V\le -\eta V$, and together with assumptions \ref{assumption:control}-(1) and \ref{assumption:control}-(2), we can obtain a uniform $\eta=\underline{w}/\bar{w}>0$. By Assumption (iii) of Theorem~\ref{resControl}, we obtain
$V_{k+1} \le V(X^+_{B_{k+1}}) e^{-\eta (t_{k+1}-t_{k})} \le \rho V(X^+_{B_{k+1}})$, where $\rho:=e^{-\eta T_{\min}}<1$. Combining the recurrences, we obtain
$V_{k+1} \le \rho (V_k + |x^*-M_k|) \le \rho V_k + |x^*-M_k|$.
As $M_k\to x^*$ and $\rho<1$, it follows that $V_k\to0$ by the discrete Gronwall inequality. Since $V(X_{B_{k+1}}(t)) \leq V(X_{B_{k+1}}^+)$ for $t \in [t_k,t_{k+1})$ (flow monocity), we obtain $\lim_{(t,k)\to \infty}V(X_{B_k}(t)) = 0$
\end{proof}

\section{Numerical Simulation}
\label{NumericalTest}
\noindent In this section, two applications\footnote{The C++ code used for the simulations is available on the GitLab repository: \scriptsize \texttt{\url{https://gitlab.lis-lab.fr/arthur.doliveira/dynsyst_vbs_code}}.} of the model $\mathcal{H}$, defined by \eqref{GenEquaDiff}–\eqref{jumpset}, are considered in order to precisely study consensus under a full connectivity assumption on the social system, with $\sigma(r) = 1$. The first one is an implementation of the model simulating consensus formation in a time-varying social structure, through the injection and deletion of agents with dynamic weights. The second one addresses the control of consensus by injecting, at regular time intervals, agents sharing the same opinion.

Regarding the first one, reported in Fig.~\ref{FigStructVar}, the initial state of the social system $X_{B_0}(0)$ is supported by the basis $B_0 = \{\mathbf{b}_p\}_{p\in\{1,2,7,8\}}$, where the opinions $x_p$ and the weights $w_p$ lie in $[0,1]$ for all $p \in I(B_0)$. The system is assumed to be fully connected, with $a_{p,q} = 1$ for all $(p,q) \in I(B_0)^2$. For the case of a fixed social structure, we can refer to the result reported in \cite[Fig.~5]{mcquade2019social}. Our results correspond to the case where the social structure varies through the addition and deletion of agents. We observe that the injection of agent 3 at time $t = 0.5$ does not affect the opinion dynamics, as the agent is injected within the convex hull of the existing opinions, but it does affect the weight dynamics. However, adding agents 4 and 5 at $t = 7.5$, and agent 6 at $t = 19$, does affect the consensus. Once consensus is achieved, deleting agent 6 at time $t = 28$ does not alter the evolution of the system. This behavior is also consistent with the evolution of the Lyapunov function plotted in Fig.~\ref{FigStructVar}.

Regarding control, the result is presented in Fig.~\ref{fig:4}, under the assumption that the weights are constant. We consider a sequence of control inputs $\{U_k\}_{k=1}^{100}\subset\mathcal{S}$, indexed at each discrete time $k$, where each $U_k$ contains one agent. All added agents are assumed to have the same opinion $x^{*}=0.8$ and a weight equal to $0.99$. The initial state of the social system $X_{B_0}(0) \in \mathcal{S}$ is supported by the basis $B_0=\{\mathbf{b}_p\}_{p=1:15}$. The agents are added once consensus has been reached, here around $0.3$, at $t\simeq8$. The objective is thus to shift this consensus toward the value $x^{*}=0.8$. The obtained result is consistent with Theorem~\ref{resControl}; in particular, the mean opinion converges toward the initial opinion $x^{*}$ of the added agents.

 \begin{figure}
    \centering
    \begin{subfigure}[b]{0.5\textwidth}
        \centering
        \includegraphics[width=0.88\textwidth]{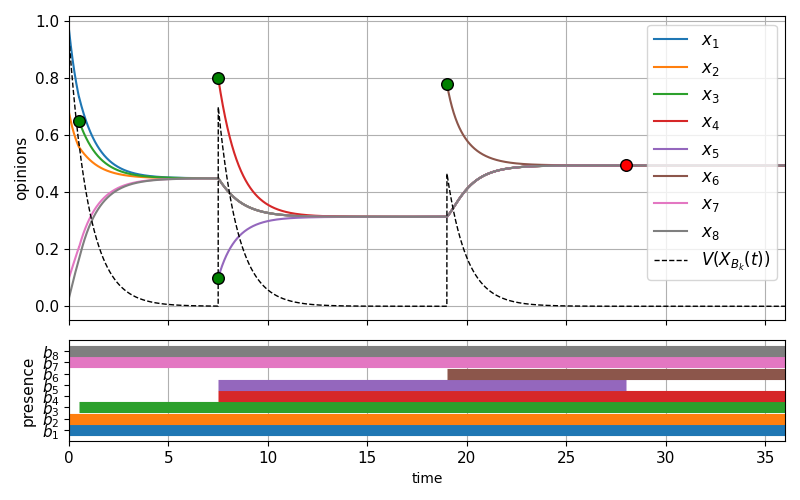}
        \caption{}
        \label{fig:sub3a}
    \end{subfigure}
    \begin{subfigure}[b]{0.5\textwidth}
        \centering
        \includegraphics[width=0.88\textwidth]{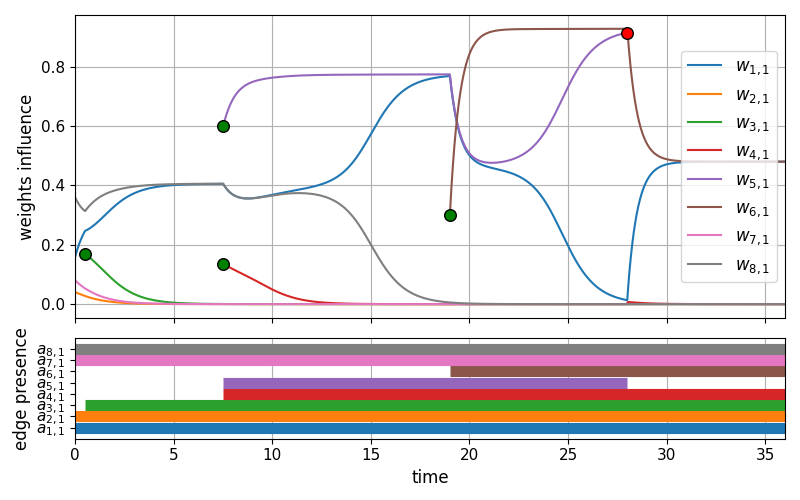}
        \caption{}
        \label{fig:sub3b}
    \end{subfigure}
    \caption{\label{FigStructVar}Simulation of the consensus process in a variable social structure: (a) opinions; (b) influence weights, with $w_{p,q} = w_p$ for all $q = 1,\dots,8$. }
\end{figure}

\begin{figure}[htbp]
    \centering
    \includegraphics[width=0.49\textwidth]{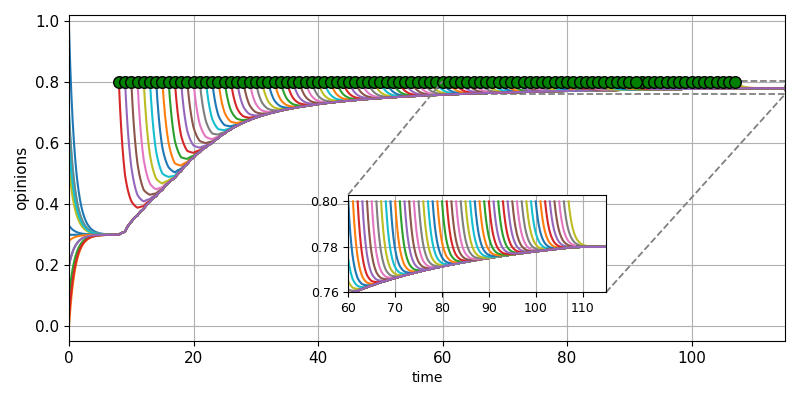}
    \caption{Consensus-based control by agent injection with the same opinion at regular time intervals, $t_{k+1}-t_k=1$.}
    \label{fig:4}
\end{figure}

\section{Conclusion}
\label{Conc}
\noindent This paper investigates opinion dynamics in a time-varying social structure characterized by the addition and deletion of agents, within the graph state space. We focus on consensus under a full connectivity assumption, considering three aspects: modeling, input-to-output stability, and consensus control through agent injection. The use of the graph state space makes it possible to incorporate, into the state trajectory of a social system, both the dynamics of opinions and the evolution of its social structure. We establish a result guaranteeing system stability through input-to-output stability, showing that the expansion of the consensus is bounded by the maximal deviation of incoming agents from the existing convex hull. In addition, we establish a consensus control result through agent injection in the case of constant weights. Future work will focus on relaxing the full connectivity assumption and developing optimal control strategies aimed at minimizing the number of injected agents. We will also investigate consensus control in the case of dynamic weights.

\bibliography{reference}
\bibliographystyle{IEEEtran}

\end{document}